\documentclass[10pt,twoside]{amsart}

\usepackage{amssymb}
\usepackage{graphicx,color}
\usepackage[all]{xy}
\usepackage{url}

\newtheorem{theorem}{Theorem}[section]

\newtheorem{lemma}[theorem]{Lemma}

\theoremstyle{definition}
\newtheorem{defn}[theorem]{Definition}
\newtheorem*{notation}{Notation}
\newtheorem{example}[theorem]{Example}

\theoremstyle{remark}
\newtheorem{remark}[theorem]{Remark}
\newtheorem{conj}[theorem]{Conjecture}

\newtheorem{question}[theorem]{Question}
\newtheorem*{acks}{Acknowledgments}

\newcommand{\ZZ}{\mathbb{Z}}
\newcommand {\PP}{\mathbb{P}}
\newcommand{\CC}{\mathbb{K}}

\newcommand{\xd}{(x_{0},x_{1}\dotsc, x_{d-1})}

\newcommand{\abs}[1]{\lvert#1\rvert}

\DeclareMathOperator{\per}{per}
\DeclareMathOperator{\Circ}{Circ}
\DeclareMathOperator{\sgn}{sgn}
\DeclareMathOperator{\LCM}{LCM}
\DeclareMathOperator{\GCD}{GCD}
\DeclareMathOperator{\Image}{Im}
\DeclareMathOperator{\Proj}{Proj}
\DeclareMathOperator{\GF}{GF}

\newcommand{\myarrow}[2]{\hbox to #1pt{\hfil$\to$\hfil}{\hskip-#1pt{\raise
10pt\hbox to#1pt{\hfil$\scriptscriptstyle #2$\hfil}}}}

\begin{document}

\title{Circulant matrices and Galois-Togliatti systems}

\author[P. De Poi]{Pietro De Poi}
\address{Dipartimento di Scienze Matematiche, Informatiche e Fisiche, Universit\`{a} degli Studi di Udine,
Via delle Scienze 206, 33100 Udine, Italy}
\email{pietro.depoi@uniud.it}

\author[E. Mezzetti]{Emilia Mezzetti}
\address{Dipartimento di Matematica e  Geoscienze, Universit\`a di
Trieste, Via A.Valerio 12/1, 34127 Trieste, Italy}
\email{mezzette@units.it}

\author[M. Micha{\l}ek]{Mateusz Micha{\l}ek}
\address{
Max Planck Institute for Mathematics in the Sciences, Inselstr.~22, 04103 Leipzig, Germany
 and
Mathematical Institute of the Polish Academy of Sciences, \'{S}niadeckich 8, 00-956
Warszawa, Poland}
\email{wajcha2@poczta.onet.pl}

\author[R. Mir\'o-Roig]{Rosa Maria Mir\'o-Roig}
\address{Facultat de Matem\`atiques i Informatica, Universitat de Barcelona, Gran Via de les Corts Catalanes 585, 08007 Barcelona,
Spain}
\email{miro@ub.edu}

\author[E. Nevo]{Eran Nevo}
\address{Institute of Mathematics, Hebrew University,
Givat Ram, Jerusalem 91904, Israel}
\email{ nevo@math.huji.ac.il}

\begin{abstract} The goal of this article is to compare the coefficients
in the expansion  of the permanent with
those
 in the expansion of the determinant of a three-lines circulant matrix.
As an application we solve a conjecture stated in \cite{MM-R2} concerning the  minimality of GT-systems.
\end{abstract}

\thanks{Acknowledgments: The first two authors are  members of INdAM - GNSAGA and are supported by PRIN
``Geometry of algebraic varieties''.  The third author was supported by the Polish National Science Centre grant no.~2015/19/D/ST1/01180. The fourth   author was partially   supported
by  MTM2016--78623-P.
The fifth  author was partially   supported
by grants ISF1695/15, ISF-NRF2528/16 and ISF-BSF2016288}

\maketitle



\date{\today}


\section{Introduction}
Circulant matrices appear naturally in many areas of mathematics. In the last decades, for instance, they have been 
related to
holomorphic mappings (\cite{D}), cryptography, coding theory (\cite{GK}),
 digital signal processing (\cite{FM}), image compression (\cite{NMG}), physics (\cite{B}), engineering simulations, number theory,  theory of statistical designs (\cite{KKNK}), etc. Even if the basic facts about these matrices can be proved in elementary way, many questions about them are subtle and remain still open (see, for instance, \cite{K-S}).

Our interest in this topic was originally motivated by its connections, exposed in \cite{MM-R2}, with a class of homogeneous ideals of a polynomial ring failing the Weak Lefschetz Property. In that context, the first question relevant to us was that of determining which monomials in the entries of a ``generic'' circulant matrix appear explicitly in the development of its determinant.

More precisely, let us denote by $\Circ \xd$  the circulant matrix of the form
\begin{equation*}
\begin{pmatrix}
x_0&x_1&x_2&\dotsb&x_{d-1}\\
x_{d-1}&x_0&x_1&\dotsb&x_{d-2}\\
x_{d-2}&x_{d-1}&x_0&\dotsb&x_{d-3}\\
\vdots&\vdots&\vdots&\ddots&\vdots\\
x_1&x_2&x_3&\dotsb&x_0
\end{pmatrix},
\end{equation*}
where $x_0, \dotsc, x_{d-1}$ are complex numbers, or more generally elements of a ring. Every summand of the determinant
$\det \Circ \xd$ is of the form
$$c_{i_0, \dotsc,i_{d-1}}x_0^{i_0}\dotsm x_{d-1}^{i_{d-1}},$$ where $c_{i_0, \dotsc,i_{d-1}}\in \ZZ$ and $i_0+\dotsm +i_{d-1}=d$. 
The question is: for which indices $i_0, \dotsc,i_{d-1}$ is the coefficient $c_{i_0, \dotsc,i_{d-1}}$ different from zero?

An analogous question can be posed for the permanent of $\Circ \xd$. In this case the answer was given in \cite{BN}, where it was proved that the monomials appearing with non-zero coefficient are precisely those whose exponents satisfy the two conditions:
\begin{equation}\label{eq:1}
\begin{cases}
i_0+\dotsm +i_{d-1}=d\\
0i_0+1i_1+2i_2+\dotsm +(d-1)i_{d-1}\equiv 0 \pmod d.
\end{cases}
\end{equation}

Clearly, conditions \eqref{eq:1} are necessary for the non-vanishing of the coefficient $c_{i_0, \dotsc,i_{d-1}}$ in  $\det \Circ \xd$. It has been recently proved that they are also sufficient if and only if $d>1$ is a power of a prime number (\cite{T}, \cite{CMMRS}).

A more general question is to find a formula for the coefficient $c_{i_0, \dotsc,i_{d-1}}$. This problem had already been considered in 1951 by Ore \cite{O}, who gave an explicit expression. Other expressions were given more recently in \cite{Mal}, \cite{WY}. However, they are not always easy to apply in order to decide if $c_{i_0, \dotsc,i_{d-1}}$ vanishes or not.

In this article, we are interested in the so-called $r$-lines circulant matrices, i.e.~circulant matrices $\Circ \xd$ of order $d>r$, where $d-r$ among $x_0, \dotsc, x_{d-1}$ are specialized to $0$. We ask if for some pairs $(r,d)$, $r<d$, conditions \eqref{eq:1} are sufficient for the non-vanishing of the corresponding coefficient in the determinant.

Our main result is
Theorem \ref{mainTHM}, where we prove that conditions \eqref{eq:1} are sufficient in the case of $3$-lines circulant matrices of order $d$, of the form 
$$\Circ (x,0,\dotsc, 0, y, 0, \dotsc, 0, z, 0, \dotsc, 0),$$ where $y$ appears in position $a$ (counting from zero), $z$ appears in position $b$, and $\GCD(a,b,d)=1$.


We also give examples of:
\begin{itemize}
\item $3$-lines circulant matrices with $\GCD(a,b,d)\neq 1$,
\item $r$-lines circulant matrices with $r\geq 4$ and similar $\GCD$ equal to one,
\end{itemize}
for which the analogous property fails.
Moreover, we prove that the coefficient of any specific monomial in a $3$-lines circulant determinant is always equal, up to the sign, to the analogous coefficient in the permanent of the same matrix under the assumption $\GCD(a,b,d)=1$.

Our results are inspired by and extend a previous result, concerning $3$-lines circulant matrices of the special form
$\Circ (x, y,  0, \dotsc, 0, z, 0, \dotsc, 0)$, given by Loehr, Warrington and Wilf~\cite{LWW}. This case, i.e.~$a=1$, was also studied by Codenotti and Resta~\cite{CR} who gave an expression for twice the permanent as a sum of four related determinants. This setting easily extends to the cases when at least one of $\GCD(a,d)$, $\GCD(b,d)$, $\GCD(b-a,d)$ equals $1$.


The second part of this article is devoted to describing applications of Theorem \ref{mainTHM}; it concerns mainly the minimality of Galois-Togliatti systems in three variables. These systems, abbreviated GT-systems, form a class of ideals of a polynomial ring introduced and studied in \cite{MM-R2}. A GT-system in three variables is the homogeneous artinian ideal $I^d_{a,b}$ of $R:=\CC[x_0, x_1, x_2]$, generated by all forms of degree $d$, that are invariant under the action of a diagonal matrix $M_{a,b}:=\begin{pmatrix}
1&0&0\\
0&e^a&0\\
0&0&e^b
\end{pmatrix}$, where $e$ is a $d$-th root of the unit.  Note that the group generated by $M_{a,b}$ is the cyclic group of order $d$ provided $\GCD(a,b,d)=1$, and that all actions of this group on $R$ can be represented by a matrix of the form $M_{a,b}$.
In \cite{MM-R2} it was proved that, if $\GCD(a,b,d)=1$, then the ideal $I^d_{a,b}$ is a Togliatti system. This means that it fails the Weak Lefschetz Property in degree $d-1$, i.e. for a general linear form $L$ (equivalently, for all linear forms), the multiplication map
$\times L\colon (R/I^d_{a,b})_{d-1}\longrightarrow (R/I^d_{a,b})_{d}$ is not injective. The authors then conjectured that $I^d_{a,b}$ is a minimal Togliatti system. As an application of Theorem \ref{mainTHM} we prove this conjecture.


 Next we outline the structure of this paper. Section \ref{Three-lines circulant matrices} contains our main results about circulant matrices. After introducing $r$-lines circulant matrices, we give a precise formulation of the problems we want to study (Question \ref{question}). We then state our main theorems (Theorems \ref{mainTHM} and \ref{thm:coef}) and produce some examples showing that our results are optimal (Examples \ref{ex:gcd} and \ref{ex:r>3}). Subsection \ref{structure} contains the proofs of the two theorems; it relies on a series of lemmas, aiming to describe the structure, in the symmetric group on $d$ elements, of the permutations that contribute non-trivially in the development of the circulant determinants we study. Section \ref{sectionGT}  is devoted to the application of the results of Section \ref{Three-lines circulant matrices} to Togliatti systems. We first recall the background about the Weak Lefschetz Property and Togliatti systems, in particular minimal monomial Togliatti systems and GT-systems, and  the conjecture on the minimality of GT-systems in three variables. Finally, Theorem \ref{thmGT}  shows how the conjecture follows from the results of Section \ref{Three-lines circulant matrices}.
In Section~\ref{sec:complexity} we indicate a computational complexity application of our main result.

\begin{notation}
Throughout this paper $\CC$ will be   an algebraically closed field of characteristic zero, $R=\CC[x_0,x_1,\dotsc ,x_n]$ and $\PP^n=\Proj(\CC[x_0,x_1,\dotsc ,x_n])$. For any polynomial $F\in R$, we denote by $[F]_{i_0,i_1,\dotsc ,i_n}$  the coefficient of the monomial $x_0^{i_0}x_1^{i_1}\dotsm x_n^{i_n}$ in $F$. Hence, we have $F=\sum_{i_0,i_1,\dotsc ,i_n}  [F]_{i_0,i_1,\dotsc ,i_n}x_0^{i_0}x_1^{i_1}\dotsm x_n^{i_n}$.
Let $S_d$ denote the symmetric group on $d$ elements.
\end{notation}

\begin{acks} This work was started at the workshop ``Lefschetz Properties and Jordan Type
in
Algebra, Geometry and Combinatorics,''
held at Levico (Trento) in June 2018.
The authors thank the Centro Internazionale per la Ricerca Matematica (CIRM) for its support.
We also thank Nati Linial and Amir Shpilka for helpful discussions on the computational complexity aspects and for pointing us to~\cite{CR}.
\end{acks}

\section{Three-lines circulant matrices}\label{Three-lines circulant matrices}

This section is devoted to the study of circulant matrices and their determinant and permanent. They have been previously studied by Ore \cite{O}, Kra and Simanca \cite{K-S},
Wyn-Jones \cite{WY} and Malenfant  \cite{Mal}; we will mainly follow Loehr, Warrington and  Wilf \cite{LWW}.
Let us start by recalling their definition:

\begin{defn}
Let $M=(y_{i,j})$ be a $d\times d$ matrix. $M$ is a \emph{circulant matrix} if,
and only if $y_{i,j}=y_{k,l}$ whenever $j-i\equiv l-k\pmod d$. That is, $M$ is
of the type
\begin{equation*}
\begin{pmatrix}
x_0&x_1&x_2&\dotsb&x_{d-1}\\
x_{d-1}&x_0&x_1&\dotsb&x_{d-2}\\
x_{d-2}&x_{d-1}&x_0&\dotsb&x_{d-3}\\
\vdots&\vdots&\vdots&\ddots&\vdots\\
x_1&x_2&x_3&\dotsb&x_0
\end{pmatrix}
\end{equation*}
where successive rows are circular permutations of the first row. It is a particular form of a Toeplitz matrix, i.e. a matrix whose elements are constant along the diagonals.
For short we denote such matrices as $\Circ_d\xd$  or simply $\Circ _d$.
\end{defn}

We now define an $r$-\emph{lines circulant matrix} as follows:
we fix an integer $r\leq d$ and an $r-$tuple of  integers
$0\leq \alpha_{0}<\dotsb<\alpha_{r-1}\leq d-1$ and  define $$\Circ_{(d;\alpha_{0},\dotsc,\alpha_{r-1})}:=(\Circ_{d})(0,\dotsc,0,x_{\alpha_{0}},0,\dotsc,0,x_{\alpha_{i}},
0,\dotsc,0,x_{\alpha_{r-1}},0\dotsc,0)$$ where $x_{\alpha _{i}}$ is located at the $\alpha _{i}+1$ position.  Notice that $\Circ_{(d;\alpha_{0},\dotsc, \alpha_{r-1})}$ is nothing but the specialization of $\Circ_{d}$ to $\{x_{i}=0\mid i\notin\{\alpha_{0},\dotsc,\alpha_{r-1}\}\}$.
Let us denote by $D_{(d;\alpha_{0},\dotsc,\alpha_{r-1})}$ (resp.~$P_{(d;\alpha_{0},\dotsc,\alpha_{r-1})}$) the number of different monomials that appear with non-zero coefficient in the expansion of  the determinant $\det(\Circ_{(d;\alpha_{0},
\dotsc,\alpha_{r-1})})$ (resp.~permanent $\per(\Circ_{(d;\alpha_{0},\dotsc,\alpha_{r-1})})$).
We always have $D_{(d;\alpha_{0},\dotsc,\alpha_{r-1})}\le P_{(d;\alpha_{0},\dotsc,\alpha_{r-1})}$ and we  are lead to pose the following questions:

\begin{question}\label{question} Fix  integers $r\leq d$ and an $r-$tuple of  integers
$0\leq \alpha_{0}<\dotsb<\alpha_{r-1}\leq d-1$.
\begin{itemize}
\item[(1)] Is $D_{(d;\alpha_{0},\dotsc,\alpha_{r-1})}=P_{(d;\alpha_{0},\dotsc,\alpha_{r-1})}$?
\item[(2)]
More strongly, comparing coefficients,
is $$[\det(\Circ_{(d;\alpha_{0},
\dotsc,\alpha_{r-1})})]_{d-A_1-\dotsm -A_{r-1},A_1,\dotsc ,A_{r-1}}=$$ $$\pm  [\per(\Circ_{(d;\alpha_{0},
\dotsc,\alpha_{r-1})})]_{d-A_1-\dotsm -A_{r-1},A_1,\dotsc ,A_{r-1}}?$$
\end{itemize}
\end{question}

In this paper, we deal with 3-lines circulant matrices.
Without loss of generality we can always assume $\alpha _0=0$ and we set $a=\alpha _1$ and $b=\alpha _2$.
We answer both questions affirmatively under the condition $\GCD(a,b,d)=1$:

\begin{theorem}\label{mainTHM} Fix integers $d\ge 3$ and $1\le a<b\le d-1$. Assume $\GCD(a,b,d)=1$. Then, $[\det(\Circ_{(d;0,a,b))})]_{d-A-B,A,B}=\pm
[\per(\Circ_{(d;0,a,b))})]_{d-A-B,A,B}$. In particular, $D_{(d;0,a,b)}=P_{(d;0,a,b)}$.
\end{theorem}
The case $a=1$ in the above theorem was established in \cite{LWW}, and our proof strategy largely follows theirs.
In fact, our proof gives the sign, as well as a combinatorial interpretation of the magnitude, for the coefficients of $\det(\Circ_{(d;0,a,b))})$:
\begin{theorem}\label{thm:coef}
Fix integers $d\ge 3$ and $1\le a<b\le d-1$ such that  $\GCD(a,b,d)=1$. Then, for any nonnegative integers $A,B$ such that $A+B\le d$,
\begin{itemize}
  
\item[(1)] $[\det(\Circ_{(d;0,a,b)})]_{d-A-B,A,B}\neq 0$ if and only if $d|(aA+bB)$.
\end{itemize}

Further, assuming $d|(aA+bB)$, then:
\begin{itemize}
\item[(2)] the sign of $[\det(\Circ_{(d;0,a,b)})]_{d-A-B,A,B}$ is $+$ if and only if at least one of $\GCD(A,B,\frac{aA+bB}{d})$ and $A+B-1$ is even; and
\item[(3)] the magnitude of $[\det(\Circ_{(d;0,a,b)})]_{d-A-B,A,B}$ equals the number of permutations in $S_d$ with cycle decomposition $C_1\circ C_2\circ\dotsm\circ C_k$, where $k=\GCD(A,B,\frac{aA+bB}{d})$, and each $C_i$ has length $A+B$ 
and consists of exactly $A$ elements $j$ with $C_i(j)\equiv j+a$ and $B$ elements $j$ with $C_i(j)\equiv j+b$, modulo $d$.
\end{itemize}
\end{theorem}

\begin{remark}\label{rem:2-lines}
For $r=2$ we may assume $\alpha_0=0$ and $\alpha_1=a$ divides $d$, with $x$  on the main diagonal and $y$ on another nontrivial diagonal of the circulant $d\times d$-matrix. One easily verifies that all coefficients of the determinant $\det(\Circ_{(d;0,a)})$ equal up to sign to the corresponding coefficients in the permanent $\per(\Circ_{(d;0,a)})$, and are given by the explicit formula
\begin{equation*}
\det(\Circ_{(d;0,a)})=
\sum_{s=0}^{a} (-1)^{(a-s)(\frac{d}{a}-1)} \binom{a}{s} x^{\frac{ds}{a}}y^{\frac{d(a-s)}{a}}.
\end{equation*}
\end{remark}
However, for $r=3$, the following example shows that the assumption $\GCD(a,b,d)=1$ cannot be dropped from Theorem \ref{mainTHM}:

\begin{example}\label{ex:gcd}  Indeed, we take $(a,b,d)=(2, 6, 12)$.
  We compute the permanent and the determinant of $\Circ_{(12;0,2,6)}$ and get:
\begin{multline*}
  \det(\Circ_{(12;0,2,6)})= x^{12}-6x^{10}y^2+15x^8y^4-20x^6y^6+15x^4y^8-6x^2y^{10}+\\
  +y^{12}-12x^8yz^3+32x^6y^3z^3-24x^4y^5z^3+4y^9z^3-2x^6z^6+42x^4y^2z^6+18x^2y^4z^6+\\
  +6y^6z^6+12x^2yz^9+4y^3z^9+z^{12}
\end{multline*}  
and
\begin{multline*}
  \per(\Circ_{(12;0,2,6)})=x^{12}+6x^{10}y^2+15x^8y^4+20x^6y^6+15x^4y^8+6x^2y^{10}+\\
  +y^{12}+12x^8yz^3+40x^6y^3z^3+48x^4y^5z^3+24x^2y^7z^3+4y^9z^3+2x^6z^6+42x^4y^2z^6+\\
  +30x^2y^4z^6+6y^6z^6+12x^2yz^9+4y^3z^9+z^{12}.
  \end{multline*}  
Therefore, we have $$D_{(12;0,2,6)}=18<P_{(12;0,2,6)}=19$$ and $$[\det(\Circ_{(12;0,2,6))})]_{6,3,3}=32 \ne   [\per(\Circ_{(12;0,2,6))})]_{6,3,3}=40.$$
\end{example}

\begin{example}\label{ex:r>3} For $r$-lines circulant  matrices with $r\ge 4$  Theorem~\ref{mainTHM} is no longer true. In fact,
\begin{enumerate}
\item  For $r=4$, we have  $D_{(6;0,2,4,5)}<P_{(6;0,2,4,5)}$ since the monomial $xz^2uv^2$
  appears in $\per(\Circ(x,0,z,0,u,v))$ but it does not appear in $\det(\Circ(x,0,z,0,u,v))$ (see \cite{CMMRS} and \cite[Example 3]{CR}).

\item Assume $r\ge 5$. We choose two prime integers $p$ and $q$ such that $p<q$ and $r\le pq$.
  Set $d=pq$. We will first prove that $D_{(d;0,1,\dotsc , d-1)}<P_{(d;0,1,\dotsc, d-1)}$.
  To this end we exhibit a $d$-tuple $A_0,A_1,\dotsc ,A_{d-1}$ such that
\begin{align*}
  A_0+2A_1+\dotsm +dA_{d-1}&\equiv 0  \pmod d\\
  A_0+A_1+\dotsm +A_{d-1}&=d
  \end{align*}
  and
\begin{equation*}
  [\det(\Circ_{(d;0,1,
    \dotsc,d-1)})]_{A_0,A_1,\dotsc ,A_{d-1}}=0.
  \end{equation*}

We apply Bezout's theorem and we write $\lambda q=1+\mu p$ with $1\le \lambda, \mu$ and $\lambda q< d$.
We define $A_0=d-\mu p-2$, $A_1=\mu p-1$, $A_{d-\mu p}=A_{\mu p -\mu \lambda +1}=A_{d-\mu p+\lambda \mu}=1$,
and $A_{i}=0$ for $i\ne 0,1, d-\mu p,\mu p -\mu \lambda +1,d-\mu p+\lambda \mu$.

By the proof of \cite[Theorem 3.5]{CMMRS}, $[\det(\Circ_{(d;0,1,
  \dotsc, d-1)})]_{A_0,A_1,\dotsc ,A_{d-1}}=0$, i.e.~the monomial
$x_0^{d-\mu p-2}x_1^{\mu p -1}x_{d-\mu p}x_{\mu p- \lambda \mu +1}x_{d-\mu p+\lambda \mu}$ appears in $\per(\Circ\xd)$ but it
does not appear in $\det(\Circ\xd)$. Therefore, $D_{(d;0,1,\dotsc, d-1)}<P_{(d;0,1,\dotsc, d-1)}$.

Since $5\le r\le d$, for any choice of an $r$-tuple $(a_0,a_1,\dotsc ,a_{r-1})$ containing
$\{0,1, d-\mu p,\mu p -\mu \lambda +1,d-\mu p+\lambda \mu \}$, the  monomial
$x_0^{d-\mu p-2}x_1^{\mu p -1}x_{d-\mu p}x_{\mu p- \lambda \mu +1}x_{d-\mu p+\lambda \mu}$ appears in the permanent of
the $d\times d$ $r$-lines circulant matrix $\Circ_{(d;a_0,a_1,\dotsc , a_{r-1})}$ but it does not appear in the determinant.
Therefore, $D_{(d;a_0,a_1,\dotsc ,a_{r-1})}<P_{(d;a_0, a_1,\dotsc ,a_{r-1})}$ and we are done.
\end{enumerate}
\end{example}

\subsection{Notation}
For a permutation $\sigma$ of $d$ elements $I:=\{0,\dotsc,d-1\}$ and an integer $q$,
we define
\begin{equation*}
  S_{q,\sigma}:=\{i\in I\mid\sigma(i)\equiv i+q \pmod d\}.
\end{equation*}
  Let:
\begin{equation*}
  P_{a,b,d,A,B}:=\{\sigma\in S_d\mid  \abs{S_{a,\sigma}}=A, \abs{S_{b,\sigma}}=B, \abs{S_{0,\sigma}}=d-A-B\}.
  \end{equation*}
  The permutations in $P_{a,b,d,A,B}$ can be characterised as those in $S_d$ where each $i$ is either fixed, i.e.~$\sigma(i)=i$,
  or translated $a$ or $b$ steps forward, i.e.~$\sigma(i)\equiv i+a$ or $\sigma(i)\equiv i+b\pmod{d}$,
  and further, the second situation $\sigma(i)\equiv i+a$ happens exactly $A$ times and the last one exactly $B$ times. Clearly,

\begin{lemma}\label{lem:per,det}
  The following equalities hold:
  \begin{align*}
[\det(\Circ_{(d;0,a,b)})]_{d-A-B,A,B}&=\sum_{\sigma\in P_{a,b,d,A,B}}\sgn(\sigma),\\
[\per(\Circ_{(d;0,a,b)})]_{d-A-B,A,B}&=\sum_{\sigma\in P_{a,b,d,A,B}}\abs{\sgn(\sigma)}=\abs{P_{a,b,d,A,B}}.
\end{align*}\qed
    \end{lemma}

    We often work with cyclic indices, say modulo $d$: an element in  $\ZZ/d\ZZ$ is uniquely determined by an element in
    $s\in I=\{0,\dotsc,d-1\}$; by abuse of notation, we frequently identify in what follows this integer $s$ with its class in  $\ZZ/d\ZZ$.
    With this identification, we will write  $s_1<s_2<s_3<\dotsb$, with $s_i\in \ZZ/d\ZZ$ if, for the
    corresponding elements in $I$ we have $0<s_2-s_1<s_3-s_1<\dotsm$.

\begin{example}
We have $4<1<2$ modulo $5$, as $0<2<3$. On the other hand it is not true that $1<3<2$.
\end{example}
\subsection{Structure of permutations}\label{structure}
In this subsection we fix a permutation $\sigma\in P_{a,b,d,A,B}$ and its canonical cycle decomposition
$\sigma=C_1\circ C_2\circ \dotsm\circ C_k$. We set $A_i=\abs{S_{a,C_i}}$, $B_i=\abs{S_{b,C_i}}$ and ``winding number'' $\ell_i=\frac{A_ia+B_ib}{d}$.
 Our aim is to prove, in steps, that the canonical cycle decomposition
$\sigma=C_1\circ C_2\circ \dotsm\circ C_k$ must be of a very special type, as described in Theorem~\ref{thm:coef}(3), and in particular all permutations $\sigma\in P_{a,b,d,A,B}$ have the same cycle structure, hence the same sign, implying Theorem~\ref{mainTHM}.

The following lemma is a straightforward generalization of \cite[Lemma 7]{LWW}, where the case $a=1$ was considered. We include the proof for the sake of completeness.
\begin{lemma}\label{lem:GCD=1}
For any integer $i$, $1\le i \le k$, we have
$\GCD(A_i,B_i,\ell_i)=1$.
\end{lemma}
\begin{proof}
Let $g:=\GCD(A_i,B_i,\ell_i)$. We consider the cycle $C_i=(s_1,\dotsc,s_w)$, where $w=A_i+B_i$ and $s_j\in \ZZ/d\ZZ$. Let $w=gw'$. To simplify notation we assume that the indexes $j$ of each $s_j$ are considered modulo $w$. We know that $s_{j+1}-s_{j}$ is congruent either to $a$ or $b$. 
 Hence, we may represent $C_i$ as a word $W$ of length $w$ with letters $a,b$:
$$W:=s_2-s_1,s_3-s_2,\dotsc,s_w-s_{w-1},s_1-s_w=:p_1\dotsb p_w.$$

\textbf{Claim:} There exists a subword $W'=p_{w_0+1},p_{w_0+2},\dotsc,p_{w_0+w'}$ of the word $W$ such that:
\begin{enumerate}
\item the number of $a$'s in $W'$ equals $A_i/g$ and
\item the number of $b$'s in $W'$ equals $B_i/g$.
\end{enumerate}
\begin{proof}[Proof of the claim:]
As the length of the word $W'$ is fixed to be $w'$, the number of letters $b$ is determined by the number of letters $a$ in it. Hence, it is enough to find $W'$ that satisfies the first condition.

The word $W$ is a concatenation of $g$ words of length $w/g$:
$$W=W_1\dotsb W_g.$$
If one of the $W_j$'s has $A_i/g$ letters $a$ the claim follows. Hence, we assume each of them has either strictly more or strictly less letters $a$ then $A_i/g$. As the total number of letters $a$ in $W$ equals $A_i$, it is not possible that all words $W_j$ have simultaneously more or simultaneously less letters $a$ then $A_i/g$. Thus, we may find two consecutive words $W_j$, $W_{j+1}$ and assume without loss of generality that $W_j$ has less and $W_{j+1}$ has more than $A_i/g$ letters $a$. Consider the sequence of all subwords of $W_jW_{j+1}$ of length $w'$ ordering them by the starting index:
$$W_j=W'_1,W'_2,\dotsc,W'_{w'+1}=W_{j+1}.$$
The word $W'_{s+1}$ is shifted by one index to the right, with respect to the word $W'_s$. In particular, $W'_{s+1}$ is getting exactly one additional letter and looses exactly one letter. 
Hence, the number of letters $a$ in $W'_s$ and $W'_{s+1}$ may differ by at most one. In particular, as the starting word $W_j=W'_1$ has less than $A_i/g$ letters $a$ and the last word $W_{j+1}=W'_{w'+1}$ has more than $A_i/g$ letters $a$, there must exist a word $W'=W'_s$ with precisely $A_i/g$ letters $a$.
\end{proof}
We may cyclically permute the entries \ $(s_1,\dotsc,s_w)$ of $C_i$. 
 By the Claim we may assume that in the multiset $\{s_2-s_1,s_3-s_2,\dotsc, s_{w'}-s_{w'-1},s_{w'+1}-s_{w'}\}$ there are precisely $A_i/g$ differences $a$ and $B_i/g$ differences $b$. Summing up all the elements of this multiset modulo $d$ we obtain:
$$s_{w'+1}-s_1=a\frac{A_i}{g}+b\frac{B_i}{g}=\frac{aA_i+bB_i}{g}=\frac{\ell_id}{g}=\frac{\ell_i}{g}d\equiv 0 \pmod d.$$
Hence, $s_{w'+1}\equiv s_1$. As the cycle $C_i$ was assumed to be primitive, the $s_j$'s must be all distinct. Hence, $w'+1=1$ modulo $w$, which means that $w'=w$ and $g=1$.
\end{proof}
\begin{lemma}\label{lem:winding1}
Suppose $\ell_1=\ell_2=1$. Then $A_1=A_2$ and $B_1=B_2$.
\end{lemma}
\begin{proof}
\textbf{First} we treat the special case when one of the numbers $A_1,A_2,B_1,B_2$ equals $0$, say $B_1=0$. As $\ell_1=1$ we have $d=a\cdot A_1$ and without loss of generality we may assume that $C_1$ consists of all numbers divisible by $a$.
If $B_2=0$ we are done. We assume $B_2\neq 0$ in order to obtain a contradiction. Let $C_2=(s_1,\dotsc,s_w)$ for $w=A_2+B_2$. As $C_1$ and $C_2$ are disjoint, we know that each $s_i$ is not divisible by $a$. 
As $a|d$ and $\GCD(a,b,d)=1$ we have $\GCD(a,b)=1$. Let $0\leq t<a$ be such that $tb\equiv -s_1$ modulo $a$. 
We have $A_2a+B_2b=d$ thus $a|B_2b$ and hence $a|B_2$. In particular $B_2\geq a>t$. Hence, there exists such $s_{i_0}$ that
$$\abs{\{1\leq i<i_0:s_{i+1}=s_i+b\}}=t,$$
i.e.~until $i_0$, the cycle $C_2$ made exactly $t$ jumps of size $b$. Modulo $a$ we have:
$$s_{i_0}=s_1+tb=s_1-s_1=0,$$
which gives the contradiction.

\textbf{Hence,} from now on we assume that all $A_1,A_2,B_1,B_2$ are nonzero. Without loss of generality we may assume $A_1+B_1\leq A_2+B_2$. Our proof is inductive on the length of the cycle $C_1$, i.e.~on $A_1+B_1$.

{\emph{Base of induction:}} $A_1+B_1=2$. As $A_1,B_1\neq 0$ we have $A_1=B_1=1$. Thus $d=a+b$. We also have $A_2a+B_2b=d=a+b$. As $A_2,B_2\neq 0$ we must have $A_2=B_2=1=A_1=B_1$.

{\emph{Inductive step:}}
We start by proving the following statement illustrated on Figure~\ref{fig:claim}.
\begin{figure}
\begin{center}
\includegraphics[scale=0.6]{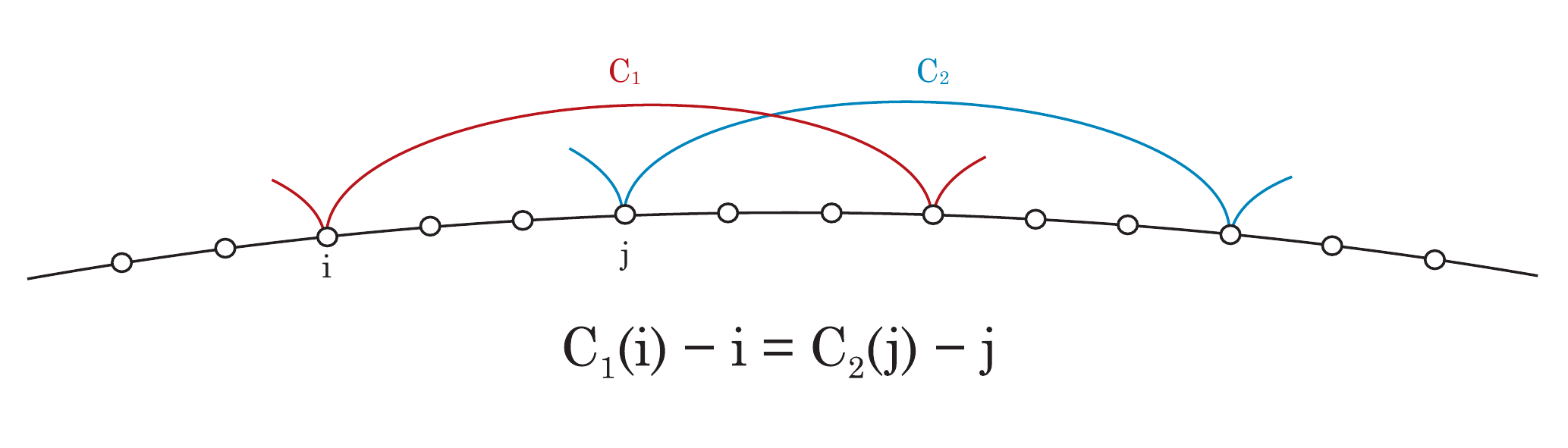}
\caption{Two cycles $C_1$ and $C_2$ with $i<j<C_1(i)$ as in the Claim in proof of Lemma~\ref{lem:winding1}.}\label{fig:claim}
\end{center}
\end{figure}
\textbf{Claim:} There exist $i,j$ such that:
\begin{itemize}
\item $C_1(i)-i=C_2(j)-j\neq 0,$
\item $i< j< C_1(i)$ or $j<i<C_2(j)$ with a cyclic ordering modulo $d$.
\end{itemize}

\begin{proof}[Proof of the claim:]
We know that $A_1,B_1\neq 0$ thus without loss of generality we may consider $i'$ such that $C_1(i')=i'+a$ and $C_1(i'+a)=i'+a+b$. Presenting $C_2=(s_1,\dotsc,s_w)$ for $w=A_2+B_2$, we may find $s_{j'}<i'+a<s_{j'+1}$.

If $s_{j'+1}-s_{j'}=a$, then setting $i=i'$ and $j=s_{j'}$ we obtain the desired indices.

If $s_{j'+1}-s_{j'}=b$, then we set $j=s_{j'}$ and $i=i'+a$ to conclude.
\end{proof}
By interchanging $C_1$ and $C_2$ and $a$ and $b$ if needed, by the above Claim we may assume that there are indices $i,j$ such that $C_1(i)-i=C_2(j)-j=b$ and
 $i< j< C_1(i)$, cf.~Figure~\ref{fig:claim}. As $\ell_1=1$, for $i<k<j$ we have $C_1(k)=k$.

\textbf{Next}, we show how to conclude in a special \textbf{Case I}, depicted on Figure~\ref{fig:c1} when there exist $i,j$ such that:
\begin{itemize}
\item for all $i<k<j$ we have $C_2(k)=k$; in other words $k$ does not belong to any of the two cycles.
\end{itemize}
\begin{figure}
\begin{center}
\includegraphics[scale=0.6]{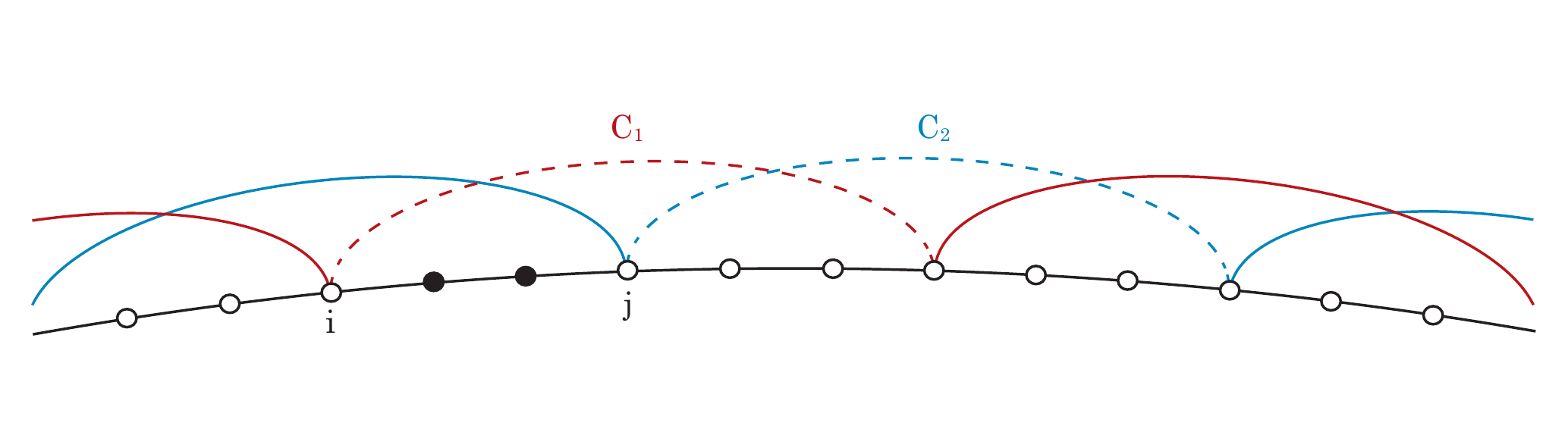}
\caption{Two cycles $C_1$ and $C_2$ with $i<j<C_1(i)$ as in Case I in proof of Lemma~\ref{lem:winding1}. The black dots do not belong to either cycle. The dashed lines indicate the parts of the cycles that will be contracted in the inductive proof.}\label{fig:c1}
\end{center}
\end{figure}


Let $d'=d-b$. We will define two disjoint cycles $C_1'$ and $C_2'$ that will be permutations of $\{1,2,\ldots,d'\}$. The cycle $C_i'$ will have exactly $A_i$ steps of size $a$ and $B_i-1$ steps of size $b$. As $\GCD(d',a,b)=\GCD(d,a,b)=1$ this will allow us to conclude by induction. Let $C_2=(s_1,\dotsc, j,j+b,\dotsc,s_w)$ for $w=A_2+B_2$ be presented in such a way that $s_1<s_2\dotsb<s_w$. We define $C_2'$ by removing $j+b$ and decreasing all indices after it by $b$: $C_2'=(s_1,\dotsc,j,\dotsc,s_w-b)$. Similarly, we present $C_1$ as $(s_1',\dotsc,i,i+b,\dotsc s'_{w'})$ for $w'=A_1+B_1$. The cycle $C_1'$ removes $i+b$ and decreases all indices after it by $b$: $C_1'=(s_1',\dotsc,i,\dotsc, s'_{w'}-b)$. The only nontrivial claim is that $C_1'$ and $C_2'$ are disjoint. Consider $C_1'$. The indices $s_1',\dotsc, i$ are distinct from $s_1,\dotsc,j$ as $C_1$ and $C_2$ were disjoint. Also $s_1',\dotsc, i$ are distinct from $s_\ell-b>j$ for $s_\ell>j$, as $i<j$. It remains to consider elements of $C_1'$ that are larger than $i$. Consider such an element $s_\ell'-b>i$. By assumption the index preceding $j$ in $C_1$ is smaller than $i$. Thus $s_\ell'-b$ could only coincide with the second part of elements of $C_2$: $j=j+b-b,\dotsc, s_w-b$. However, $s_\ell'-b= s_q-b$ would imply that $s_\ell'=s_q$, which is not possible. This finishes the proof in the special Case I.
\begin{figure}
\begin{center}
\includegraphics[scale=0.6]{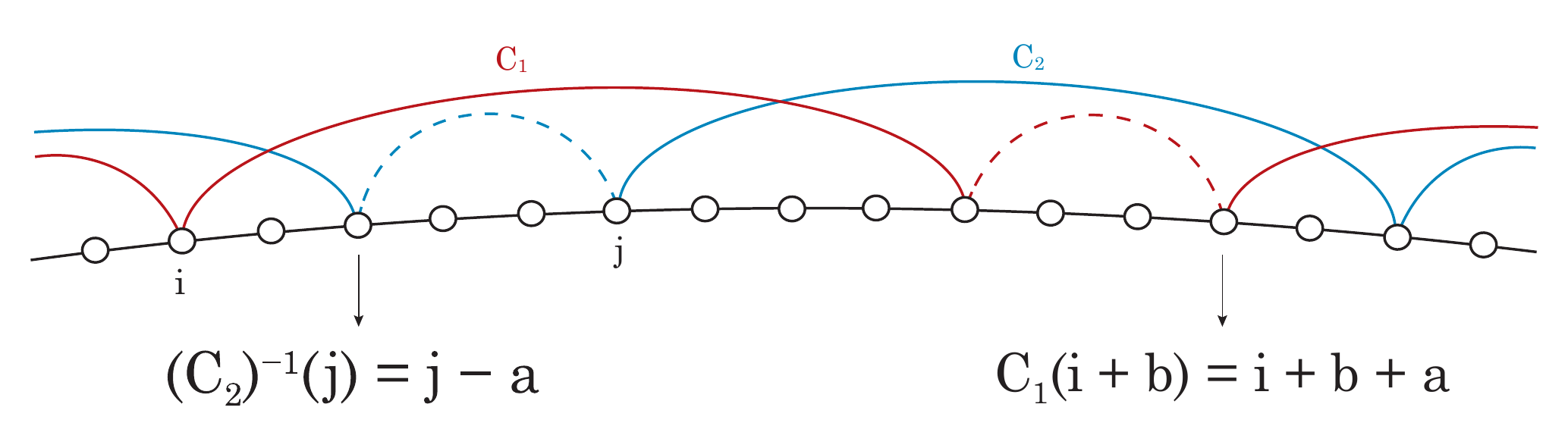}
\caption{Two cycles $C_1$ and $C_2$ with $i<j<C_1(i)$ as in Case II in proof of Lemma~\ref{lem:winding1}. The dashed lines indicate the parts of the cycles that will be contracted in the inductive proof.}\label{fig:c2}
\end{center}
\end{figure}

\textbf{Case II}:
If we are not in the special Case I, we must have  $i<(C_2)^{-1}(j)<j$, cf.~Figure~\ref{fig:c2}. This implies $j-(C_2)^{-1}(j)=a$ and $a<b$. By inverting the direction of both cycles we see that we may also assume that $C_1(i+b)=i+b+a$. In analogy to Case I, we set $d'=d-a$ and we will define two cycles $C_1'$ and $C_2'$. Precisely if $C_2=(s_1,\dotsc,j-a,j,j+b,\dotsc, s_w)$ let $C_2'=(s_1,\dotsc,j-a,j+b-a,\dotsc, s_w-a)$,  and if $C_1=(s_1',\dotsc,\textbf{i},i+b,i+b+a,\dotsc,s_{w'})$ let $C_1'=(s_1',\dotsc,i,i+b,\dotsc,s_{w'}-a)$. It is straightforward to check, as in the Case I, that the cycles $C_1'$ and $C_2'$ are disjoint. We conclude by induction.
\end{proof}
\begin{remark}
Lemma \ref{lem:winding1} may not hold when $\GCD(a,b,d)\neq 1$. Indeed, consider $d=8$, $a=2$, $b=4$. Let $\sigma=(0,2,4,6)\circ (1,5)$. Clearly the two cycles in the decomposition do not satisfy the conclusion of the lemma.
\end{remark}
\begin{lemma}\label{lem:relOFcycles}
For any $1\leq i,j\leq k$ we have $A_i=A_j$ and $B_i=B_j$.
\end{lemma}
\begin{proof}
Without loss of generality we assume $i=1$ and $j=2$.

Our aim is to rearrange the cycles $C_1$ and $C_2$, possibly changing $a,b,d$, in such a way that we can apply Lemma \ref{lem:winding1}.

\textbf{Reduction 1:} We first reduce to the situation when $a|d$ and $\GCD(a,b)=1$. Precisely, we construct new cycles, without changing $A_1,A_2,B_1,B_2$ and $d$, however with new $a'$ and $b'$, such that $\GCD(a',b')=1$ and $a'|d$.

Let $a':=\GCD(a,d)$ and $a=za'$.
Let $z'<d$ be such that $z\cdot z'\equiv 1$ modulo $d$. Let $d>b':=bz'$ modulo $d$. We rearrange the rests $\{0,1,\dotsc,d-1\}$ modulo $d$ by multiplying them by $z'$. Precisely, we change the cycle $C_1=(s_1,\dotsc,s_{A_1+B_1})$ and $C_2= (s_1',\dotsc,s_{A_2+B_2}')$ respectively to $C_1'=(z's_1,\dotsc,z's_{A_1+B_1})$ and $C_2'= (z's_1',\dotsc,z's_{A_2+B_2}')$. These are clearly two disjoint cycles with possible differences $a'$ and $b'$. Further $a'|d$ and $\GCD(a',b',d)=\GCD(a',bz',d)$. Note that $\GCD(a',b,d)=1$ (as $a'|a$) and $\GCD(a',z',d)=1$ (as $z'$ and $d$ are coprime), hence $\GCD(a',b',d)=1$.

\textbf{Reduction 2:}
Let $C_i'$ be as above, $i=1,2$.
We now reduce to the case of winding numbers $\ell_1=\ell_2=1$. Let $m:=\GCD(\ell_1,\ell_2)$ and $\ell_1=m\ell_1'$, $\ell_2=m\ell_2'$.

Let $d'=dm\ell_1'\ell_2'=d\LCM(\ell_1,\ell_2)$. We extend the cycle $C_1'$ (that is a permutation of $d$ elements) to a cycle $C_1''$ (that is a permutation of $d'$ elements) as follows. Say $C_1'=(c_1,\dotsc,c_{A_1+B_1})$, where $c_1$ is the smallest integer appearing in $C_1'$. We may encode it as a word, with letters $a'$ and $b'$, of length $A_1+B_1$:
$$w:=c_2-c_1,\dotsc,c_{A_1+B_1}-c_{A_1+B_1-1}.$$
We concatenate the word $w$ with itself $\ell_2'$ times, obtaining a word $w^{\circ \ell_2'}$ with exactly $A_1\ell_2'$ letters $a'$ and $B_1\ell_2'$ letters $b'$. The word $w^{\circ \ell_2'}$ encodes the cycle $C_1''$, that starts at $c_1$ with differences $c_{i+1}-c_i$ equal to either $a'$ or $b'$, according to $w^{\circ \ell_2'}$. In the analogous way we obtain a cycle $C_2''$ with $A_2\ell_1'$ differences $a'$ and $B_2\ell_2'$ differences $b'$.

We are now in position to apply Lemma \ref{lem:winding1}. As $a'|d$, we must have $\GCD(a',b')=1$. In particular, $\GCD(a',b',d')=1$. Further, the cycles $C_1''$ and $C_2''$ have their winding numbers $\ell_1''=\ell_2''=1$. They are also disjoint, as their reductions modulo $d$ coincide with $C_1$ and $C_2$ that are disjoint. Hence, by Lemma~\ref{lem:winding1}, $A_1\ell_2'=A_2\ell_1'$ and $B_1\ell_2'=B_2\ell_1'$.


In particular, $\ell_1'|A_1\ell_2'$ and as $\GCD(\ell_1',\ell_2')=1$ we have $\ell_1'|A_1$. In the same way $\ell_1'|B_1$ and by definition $\ell_1'|\ell_1$. Thus, by Lemma \ref{lem:GCD=1} we must have $\ell_1'=1$. Analogously $\ell_2'=1$. Hence, indeed $A_1=A_2$ and $B_1=B_2$.
\end{proof}


Recall $A=\sum_{i=1}^{k} A_i$, $B=\sum_{i=1}^{k} B_i$, and let $\ell:=\frac{aA+bB}{d}$.
\begin{lemma}
We have $k=\GCD(A,B,\ell)$.
\end{lemma}
\begin{proof}
By Lemma \ref{lem:relOFcycles} we know that $A=kA_1$, $B=kB_1$ and $\ell=k\ell_1$. We conclude by Lemma \ref{lem:GCD=1}:
$$k=k\GCD(A_1,B_1,\ell_1)=\GCD(kA_1,kB_1,k\ell_1)=\GCD(A,B,\ell).$$
\end{proof}
By Lemma~\ref{lem:per,det}, the results of this subsection imply Theorems~\ref{mainTHM} and~\ref{thm:coef}. $\square$

\section{On the minimality of GT-systems}\label{sectionGT}

In this section, we will apply the results on the determinant of a three-lines circulant matrix obtained in the previous section  to study the minimality of GT-systems and to solve a conjecture stated by Mezzetti and Mir\'{o}-Roig in \cite{MM-R2}. To state this conjecture we need first to introduce some definitions.

\begin{defn}\label{wlp}
Let $I\subset R $ be a homogeneous artinian ideal. We say that   $I$ has the \emph{Weak Lefschetz Property} (WLP, for short)
 if there is a  $L \in [R/I]_1$ such that, for all
integers $j$, the multiplication map
\[
\times L\colon [R/I]_{j-1} \to [R/I]_j
\]
has maximal rank, i.e. it is either injective or surjective.
\end{defn}
To establish whether an ideal $I\subset R$ has the WLP is a  difficult and  challenging problem
 and even in simple
cases, such as complete intersections, much remains unknown about the presence of the WLP.
 Recently the failure of the WLP has been connected to a large number of problems,
that appear to be unrelated at first glance.
For example, in \cite{MM-RO}, Mezzetti, Mir\'{o}-Roig and Ottaviani proved
that the failure of the WLP is related to the existence of varieties satisfying at least one Laplace
equation of order greater than 2; we recall that a $k$-dimensional variety $X\subset\PP^n$
\emph{satisfies $r$ Laplace equations of  order $d$}
if for any parametrization $F=F(t_1,\dotsc,t_k)$ of $X$ around a smooth, general point,  $F$ satisfies a system of $r$
(linearly independent) PDE's with constant coefficients of order $d$. Their result is the following:

\begin{theorem}(\cite[Theorem 3.2]{MM-RO})\label{tea} Let $I\subset R$ be an artinian
ideal
generated
by $r$ homogeneous polynomials $F_1,\dotsc,F_{r}$ of degree $d$ and let $I^{-1}$ be its Macaulay inverse system.
If
$r\le \binom{n+d-1}{n-1}$, then
  the following conditions are equivalent:
\begin{itemize}
\item[(1)] the ideal $I$ fails the WLP in degree $d-1$;
\item[(2)] the  homogeneous forms $F_1,\dotsc,F_{r}$ become
$k$-linearly dependent on a general hyperplane $H$ of $\PP^n$;
\item[(3)] the $n$-dimensional   variety
 $X=\overline{\Image (
\varphi )}$
where
$\varphi \colon\PP^n \dashrightarrow \PP^{\binom{n+d}{d}-r-1}$ is the  rational map associated to $(I^{-1})_d$,
  satisfies at least one Laplace equation of order
$d-1$.
\end{itemize}
\end{theorem}


The above result motivated the following definitions:

\begin{defn} Let $I \subset R$ be an artinian ideal generated by $r$ forms  of degree $d$, and $r \leq \binom{n+d-1}{n-1}$. We will say:
\begin{itemize}
\item[(i)] $I$ is a \emph{Togliatti system} if it satisfies one of three equivalent conditions in Theorem \ref{tea}.
\item[(ii)] $I$ is a \emph{monomial Togliatti system} if, in addition, $I$ can be generated
by monomials.
\item[(iii)] $I$ is a \emph{smooth Togliatti system} if, in addition, the rational variety $X$ is smooth.
\item[(iv)] A monomial Togliatti system $I$ is \emph{minimal} if  there is no proper subset of the set of generators defining a monomial Togliatti system.
\end{itemize}
\end{defn}

These definitions were introduced in \cite{MM-RO} and \cite{MM-R1} and the names are in honor of Eugenio Togliatti who proved that for
$n = 2$ the only smooth
Togliatti system of cubics is $$I = (x_0^3,x_1^3,x_2^3,x_0x_1x_2)\subset \CC[x_0,x_1,x_2]$$
(see \cite{BK}, \cite{T1} and \cite{T2}). The systematic study of Togliatti systems was initiated in \cite{MM-RO}
and for recent results the reader can see \cite{MichM-R},  \cite{MM-R1}, \cite{AM-RV}, \cite{M-RS} and \cite{MM-R2}.
Precisely in the latter reference the authors introduced the notion of \emph{GT-system} which we recall now.

\begin{defn}\label{Defi:GT}
  A \emph{GT-system} is an artinian ideal $I\subset \CC[x_0, x_1,\dotsc ,x_n]$ generated by $r$ forms $F_{1},\dotsc,F_{r}$ of degree $d$
  such that:
\begin{enumerate}
\item[i)] $I$ is a Togliatti system.
\item[ii)] The regular map $\phi_{I}\colon\PP^{n}\rightarrow \PP^{r-1}$ defined by $(F_{1},\dotsc,F_{r})$ is a Galois covering of degree $d$ with cyclic Galois group $\ZZ /d\ZZ$.
\end{enumerate}
\end{defn}

Any representation of the cyclic group $\ZZ  /d\ZZ$ as subgroup of $GL(n+1,\CC)$ can be diagonalized.
In particular it is represented by a matrix of the form
$$M:=M_{\alpha_0,\alpha_1,\dotsc,\alpha_{n}}=\begin{pmatrix} e^{\alpha _0} & 0 & \dotsc & 0 \\ 0 & e^{\alpha _1} &  \dotsc & 0 \\
 &  & \dotsc &  \\ 0 & 0 &\dotsc & e^{\alpha _{n}}
\end{pmatrix}$$
where $e$ is a primitive $d$th root of $1$ and $\alpha_0,\alpha_1,\dotsc,\alpha_{n}$ are integers with
$$\GCD(\alpha_0,\alpha_1,\dotsc,\alpha_{n},d)=1.$$ 
It follows (see \cite[Proposition 4.6]{CMMRS}) that the above definition is equivalent to the next one:

\begin{defn}
Fix integers $3\le d\in \ZZ $, $2\le n\in \ZZ $, with $n\leq d$, and $0\le \alpha_0\le \alpha_1\le \dotsb \le \alpha_{n}\le d$, $e$ a primitive $d$-th root of 1 and $ M_{\alpha_0,\alpha_1,\dotsc,\alpha_{n}}$ a representation of $\ZZ /d\ZZ$ in $GL(n+1,\CC)$. A \emph{GT-system} will be an ideal $$I^d_{\alpha_0,\dotsc,\alpha _{n}}\subset \CC[x_0,x_1,\dotsc,x_{n}]$$ generated by all forms of degree $d$ invariant under the action of   $ M_{\alpha_0,\alpha_1,\dotsc,\alpha_{n}}$ provided the number of generators $\mu (I^d_{\alpha_0,\dotsc,\alpha _{n}})\le  \binom{n+d-1}{n-1}$.
\end{defn}

\begin{remark}
It is an immediate consequence of the above description that the ideal $I^d_{\alpha_0,\dotsc,\alpha_n}$ is always monomial,
i.e. a GT-system is a monomial Togliatti system.
\end{remark}

\begin{remark}\label{remark:minim}
Indeed in the proof of \cite[Proposition 4.6]{CMMRS}, the authors observed that $I:=I^d_{\alpha_0,\dotsc,\alpha _{n}}$ fails the Weak Lefschetz Property from degree $d-1$ to degree $d$, because
for any linear form $\ell \in R$ the induced map $\times \ell :[R/I]_{d-1}\longrightarrow [R/I]_d$ is not injective.
By \cite[Proposition 2.2]{MM-RN},  since $I$ is monomial, it is enough to check it for $\ell = x_0+x_1+\dotsm +x_{n}$.
This is equivalent to prove that   there exists a form
$F_{d-1}\in R$ of degree $d-1$ such that $(x_0+x_1+\dotsm +x_{n})\cdot F_{d-1}\in I$.
Consider $F_{d-1}=(e^{\alpha _0}x_0+e^{\alpha _1}x_1+\dotsm
+e^{\alpha _{n}}x_{n})(e^{2\alpha _0}x_0+e^{2\alpha _1}x_1+\dotsm +e^{2\alpha _{n}}x_{n})
\dotsc
(e^{(d-1)\alpha _0}x_0+e^{(d-1)\alpha _1}x_1+\dotsm +e^{(d-1)\alpha _{n}}x_{n}).$
The homogeneous form of degree $d$, $F=(x_0+x_1+\dotsm +x_{n})\cdot F_{d-1}$ is invariant under the action of
$M_{\alpha_0,\alpha_1,\dotsc,\alpha_{n}}$, hence, it belongs to $I$.
\end{remark}

In the following, since we are interested in the projective space and in the action of $M_{\alpha_0,\alpha_1,\dotsc,\alpha_{n}}$ up to proportionality, we will assume that the first exponent $\alpha_0$ is equal to zero. Moreover, $R$ will always denote the polynomial ring in three variables: $R=\CC[x,y,z]$.

To determine the minimality of a GT-system is a subtle problem. In \cite{MM-R2}, in the case of three variables, the second and fourth authors proved that the ideal $I^d_{0,a,b}$ always satisfies the condition on the number of generators  $\mu(I)\leq d+1$, and conjectured the following, which we now prove using Theorem~\ref{mainTHM}:

\begin{theorem}(\cite[Conjecture 4.6]{MM-R2})\label{thmGT}
Let $d\geq3$ be an integer and $M_{a,b}$ be a $3\times 3$-matrix representing the cyclic group $\ZZ/d\ZZ$ with $1\leq a< b\leq d-1$ such that $\GCD(a,b,d)=1$. Let $I=I^d_{0,a,b}\subset R=\CC[x,y,z]$ be the ideal generated by all the monomials of degree $d$ invariant under the action of $M_{a,b}$. Then $I$ is a minimal GT-system.
\end{theorem}

\begin{proof}
As observed in Remark \ref{remark:minim}, the form $$F_{d-1}:=(x+e^{a}y+e^{b}z)\dotsb(x+e^{(d-1)a}y+e^{(d-1)b}z)$$
 is in the kernel of $\times(x+y+z):[R/I]_{d-1}\rightarrow[R/I]_{d}$,
so the dimension of the kernel 
is $\ge 1$, thus $I$ is a GT-system.
We now work towards showing its minimality.

\noindent \textbf{Claim:} The dimension of the  kernel $K_{d-1}$ of $\times (x+y+z):[R/I]_{d-1}\rightarrow[R/I]_{d}$ is one.

\begin{proof}[Proof of the claim:]
We will prove that $F_{d-1}$ generates  $K_{d-1}$. 

Assume that $G_{d-1}(x,y,z)$ is a form of degree $d-1$
which  belongs to $K_{d-1}$. We will prove that $(x+e^ay+e^bz)$ divides
$G_{d-1}(x,y,z)$. Analogously  the other factors of $F_{d-1}$
divide $G_{d-1}(x,y,z)$, and we are done. Since $G_{d-1}(x,y,z)$ belongs to
$K_{d-1}$, we have $(x+y+z)G_{d-1}(x,y,z) \in I$. So the  form $H(x,y,z):=(x+y+z)G_{d-1}(x,y,z)$ of degree $d$ is in  $I$.

Since $H(x,y,z)$ belongs to $I$, it is invariant under the
action of $M_{a,b}$ and we have

\begin{equation*}(x+e^ay+e^bz)G_{d-1}(x,e^ay,e^bz)=H(x,e^ay,e^bz)=H(x,y,z)=(x+y+z)G_{d-1}(x,y,z),
\end{equation*}

\noindent which allows us to conclude that $(x+e^ay+e^bz)$ divides $G_{d-1}(x,y,z)$.
  \end{proof}

Our next claim is that the Togliatti system $I$ is minimal if and only if all  monomials of degree $d$, which are invariant under the action of $M_{a,b}$,  appear with non-zero coefficient in the form \begin{equation}\label{Eq:circulant}
 C_{d;a,b}:=(x+y+z)(x+e^{a}y+e^{b}z)\dotsb(x+e^{(d-1)a}y+e^{(d-1)b}z)=(x+y+z)F_{d-1}.
 \end{equation} One implication is obvious. For the other,
assume that $I$ is not a minimal Togliatti system: this means that there is an ideal $J$, strictly contained in $I$, which is again a Togliatti system. Let $G_1, \dotsc, G_s$ be a system of generators of $J$. Then for any linear form $\ell$ there is a form $G$ such that $\ell G$ is a linear combination of  $G_1,\dotsc, G_s$.
In particular, $(x+y+z) G$ belongs to $I$, therefore $G$ is in the kernel of the map $\times (x+y+z)\colon [R/I]_{d-1}\rightarrow[R/I]_{d}$. Since the kernel has dimension one, by the Claim, it follows that $G$ is a scalar multiple of $F_{d-1}$. We conclude that not all invariant monomials appear with non-zero coefficient in $C_{d;a,b}$.

We easily observe that $C_{d;a,b}$ coincides with the determinant of the circulant matrix $\Circ_{(d;0,a,b)}$. On the other hand, a monomial $x^{\alpha}y^{\beta}z^{\gamma}$ of degree $d$ is invariant under the action of $M_{a,b}$ if, and only if, it satisfies the following system of equations:
\begin{equation*}
\begin{cases}
\alpha+\beta+\gamma=d&\\
a\beta+b\gamma\equiv0&\pmod d
\end{cases}
\end{equation*}
or, equivalently, the monomials appearing in the permanent $\per(\Circ_{(d;0,a,b)})$ with non-zero coefficient are exactly all the monomials of degree $d$ invariant by the action of $M_{a,b}$.



Thus, to conclude the proof we need the equality $D_{(d;0,a,b)}=P_{(d;0,a,b)}$ to hold, which is indeed the case by Theorem \ref{mainTHM}.
\end{proof}

\begin{remark} It is worthwhile to underline the following interpretation of our results in terms of representation theory of cyclic matrix groups.

In the proof of Theorem \ref{thmGT} we have proved that a monomial  of degree $d$ in three variables is invariant under the action of the cyclic matrix group of order $d$, generated by  $M_{a,b}$,  if and only if it appears with non-zero coefficient in the form $C_{d;a,b}$ of (\ref{Eq:circulant}).
So we get information about a minimal system of generators for the homogeneous component of degree $d$ of the ring of invariants of these representations.

Up to a coefficient $d$, the polynomial $C_{d;a,b}$ is the image of the linear form $x+y+z$ under the Reynolds operator. This is in fact the same point of view of Emmy Noether, when she proved the finiteness of the ring of invariants of a polynomial ring under the action of a finite matrix group \cite{N}.

From Example \ref{ex:r>3} it follows that a similar result is not true for polynomials in $r$ variables, with $r\geq 4$.
\end{remark}

\section{Computational complexity}\label{sec:complexity}
The task of computing the permanent of a $d$ by $d$ $(0/1)$-matrix is computationally hard ($\#P$-complete), by a celebrated theorem of Valiant~\cite{V}, and remains so even when there are only $3$ nonzero entries per row~\cite{DLMV}.
The best known upper bounds to compute the permanent are exponential in $d$, by Ryser~\cite{R}. Theorem~\ref{thm:coef} immediately tells us that computing the permanent of $\Circ_{(d;0,a,b)}$, when $\GCD(a,b,d)=1$, can be done in polynomial time in $d$, say by computing $\det(\Circ_{(d;0,a,b)})\in \mathbb{Z}[x,y,z]$ via polynomial interpolation (evaluating it on $O(d^2)$ points suffices).


\end{document}